\documentclass[a4paper,11pt]{article}

\usepackage{amssymb}
\usepackage{color}
\usepackage{comment}
\usepackage{enumerate}
\usepackage{url}
\usepackage{amsmath}  
\usepackage{multirow}
\usepackage{lineno, blindtext}
\usepackage{changepage}
\usepackage[nottoc]{tocbibind}
\usepackage{rotating}
\usepackage{makecell}
\usepackage{authblk}
\usepackage{mathtools}
\usepackage{caption}
\usepackage{subcaption}
\usepackage{amsthm}
\usepackage{graphicx}

\usepackage[
bookmarks=true,
bookmarksnumbered=true, 
bookmarkstype=toc,
setpagesize=false,
pdftitle={Polyhedral-based Methods to Mixed-Integer SOCP in Tree Breeding},
pdfauthor={Sena Safarina, Tim J. Mullin, Makoto Yamashita},
pdfkeywords={mixed-integer second-order cone programming, tree breeding, cone decomposition}
]{hyperref}

\DeclarePairedDelimiter{\ceil}{\lceil}{\rceil}
\def\@themcountersep{}
\newtheorem{THEO}{Theorem}[section]
\newtheorem{ALGO}[THEO]{Algorithm}

\newtheorem{CORO}[THEO]{Corollary}
\newtheorem{DEFI}[THEO]{Definition}

\def\0{\mbox{\bf 0}}
\def\1{\mbox{\bf 1}}
\def\2{\mbox{\bf 2}}
\def\3{\mbox{\bf 3}}
\def\4{\mbox{\bf 4}}
\def\5{\mbox{\bf 5}}
\def\6{\mbox{\bf 6}}
\def\7{\mbox{\bf 7}}
\def\8{\mbox{\bf 8}}
\def\9{\mbox{\bf 9}}
\def\a{\mbox{\boldmath $a$}}

\newdimen\zhige \zhige=0pt
\def\chige#1{{\setbox\zhige\hbox{#1}\ifdim\ht\zhige=1ex\accent24 #1%
  \else\ooalign{\unhbox\zhige\crcr\hidewidth\char24\hidewidth}\fi}}

\def\e{\mbox{\boldmath $e$}}

\def\g{\mbox{\boldmath $g$}}

\def\l{\mbox{\boldmath $l$}}

\def\u{\mbox{\boldmath $u$}}
\def\v{\mbox{\boldmath $v$}}
\def\w{\mbox{\boldmath $w$}}
\def\x{\mbox{\boldmath $x$}}
\def\y{\mbox{\boldmath $y$}}
\def\z{\mbox{\boldmath $z$}}
\def\A{\mbox{\boldmath $A$}}

\def\H{\mbox{\boldmath $H$}}

\def\U{\mbox{\boldmath $U$}}

\def\W{\mbox{\boldmath $W$}}

\def\CC{\mbox{$\cal C$}}

\def\KC{\mbox{$\cal K$}}

\def\PC{\mbox{$\cal P$}}

\def\Real{\mbox{$\mathbb{R}$}}

\ifx11 
  \newcommand{\red}[1]{\textcolor{red}{#1}}

\fi

\title{Polyhedral-based Methods for Mixed-Integer SOCP in Tree Breeding}
\author[1]{Sena Safarina}
\author[2]{Tim J. Mullin}
\author[1]{Makoto Yamashita} 
\affil[1]{Department of Mathematical and Computing Science,
 Tokyo Institute of Technology, 2-12-1-W8-29 Ookayama, Meguro-ku, Tokyo
 152-8552, Japan.}                    
 \affil[2]{
 The Swedish Forestry Research Institute (Skogforsk), Box 3, S\"{a}var 918 21, Sweden; 
and 224 rue du Grand-Royal Est, QC, J2M 1R5, Canada.}
\begin{document}
\maketitle              

\begin{abstract}
Optimal contribution selection (OCS) is a mathematical optimization problem that aims to maximize the total benefit from selecting a group of individuals under a constraint on genetic diversity. We are specifically focused on OCS as applied to forest tree breeding, when selected individuals will contribute equally to the gene pool. Since the diversity constraint in OCS can be described with a second-order cone, equal deployment in OCS can be mathematically modeled as mixed-integer second-order cone programming (MI-SOCP). If we apply a general solver for MI-SOCP, non-linearity embedded in OCS requires a heavy computation cost. To address this problem, we propose an implementation of lifted polyhedral programming (LPP) relaxation and a cone-decomposition method (CDM) to generate effective linear approximations for OCS. In particular, CDM successively solves OCS problems much faster than generic approaches for MI-SOCP.
The approach of CDM is not limited to OCS, so that we can also apply the approach to other MI-SOCP problems.
\end{abstract}


\section{Introduction}
As in other types of breeding, forest tree improvement is based on recurrent cycles of selection, mating and testing. In the selection phase, we should take genetic diversity into consideration so that tree health and the potential for genetic gain in the future are conserved. A general objective of optimal contribution selection (OCS)~\cite{ahlinder2014using, meuwissen2002gencont, mullin2016using, yamashita2017efficient} is to maximize the total economic benefit under a genetic diversity constraint by determining the gene contribution to be made from each candidate. Based on the type of contribution, OCS problems can be classified into unequal and equal deployment problems. While an unequal deployment problem (UDP) does not require the same contribution for selected candidates, an equal deployment (EDP) stipulates that a specified number of selected individuals must contribute equally to the gene pool.

A mathematical optimization formulation for UDP is given by Meuwissen \cite{meuwissen1997maximizing} as follows:
\begin{eqnarray}
\begin{array}{lcl}
\label{eq:QCP}
\mbox{maximize} &:&  \g^T\x \\
\mbox{subject to} &:& \e^T \x = 1, \\
& & \l \leq \x \leq \u, \\
& & \x^T\A\x \le 2\theta.
\end{array}
\end{eqnarray}
The decision variable is $\x \in \Real^m$ that corresponds to the gene contributions of individual candidates, where $m$ is the total number of candidates. The objective is to maximize the total benefit $\g^T\x$ where the vector $\g=\{g_1, g_2, \ldots,g_m\}$ contains the estimated breeding values (EBVs)~\cite{lynch1998genetics} representing the genetic value of candidates in $\x$. In this paper, we assume that $\g$ is given. Using a vector of ones $\e\in\Real^m$, the constraint $\e^T\x=1$ requires that the total contribution of all candidates be unity. The next constraint is composed by a lower bound $\l\in\Real^m$ and an upper bound $\u\in\Real^m$.

The crucial constraint in (\ref{eq:QCP}) is $\x^T\A\x \le 2\theta$ that requires the group coancestry $\frac{\x^T \A \x}{2}$ be under an appropriate level $\theta\in\Real_{++}$, where $\mathbb{R}_{++}$  is the set of positive real numbers. The group coancestry constraint $\x^T\A\x \le 2\theta$ was originally introduced by Cockerham \cite{cockerham1967group}, while the construction of the numerical relationship matrix $\A\in\Real^{m\times m}$ was proposed in Wright~\cite{wright1922coefficients}.
Shortly speaking, each element $A_{ij}$ in the matrix $A$ is the probability that genotypes $i$ and $j$ have a common ancestor.
Pong-Wong and Woolliams~\cite{pong2007optimisation} observed that the matrix $\A$ is always positive definite, and they formulated the UDP as a semi-definite programming (SDP) problem. Their SDP approach gave the exact optimal value to the UDP for the first time, but Ahlinder~\cite{ahlinder2014using} reported that the computation cost of the SDP approach was very high, even when using a parallel SDP solver~\cite{SDP-HANDBOOK, SDPA6}. To reduce the heavy computation burden, Yamashita et al.~\cite{yamashita2017efficient} proposed an efficient numerical method that exploits the sparsity in the inverse matrix $\A^{-1}$.
Their method is based on second-order cone programming (SOCP) \cite{alizadeh2003second}.

The current research is mainly concerned with EDP of form:
\begin{eqnarray}
\begin{array}{lcl}
\label{eq:MI-QCP}
\mbox{maximize} &:&  \g^T\x \\
\mbox{subject to} &:& \e^T \x = 1, \\ 
& & x_i \in \left\{0,\frac{1}{N}\right\} \ \mbox{for} \ i=1,\ldots,m, \\
& & \x^T\A\x \le 2\theta.
\end{array}
\end{eqnarray}
We should emphasize that the simple bound $\l \le \x \le \u$ in the UDP is replaced by another constraint $x_i \in \left\{0,\frac{1}{N}\right\}$ to require an equal contribution from each selected candidate. Here, $N$ is the parameter to indicate the number of chosen candidates. In short, we have to choose exactly $N$ individuals from a list of $m$ available candidates in the EDP. Through this paper, we assume that \eqref{eq:MI-QCP} is feasible.

The OCS problem has been widely solved through a software package \texttt{GENCONT} developed by Meuwissen~\cite{meuwissen2002gencont}.
The numerical method implemented in \texttt{GENCONT} is based on Lagrange multipliers, but it forcibly fixes variables that exceed lower or upper bounds $\left(0\le x_i\le\frac{1}{N}\right)$ at the corresponding lower and upper bound. Thus, even though \texttt{GENCONT} generates a solution quickly, the solution is often suboptimal. To resolve this difficulty in \texttt{GENCONT}, another tool \texttt{dsOpt}, incorporated in the software package \texttt{OPSEL} \cite{mullin2013opsel}, was proposed by Mullin and Belotti~\cite{mullin2016using}. \texttt{dsOpt} is an implementation of the branch-and-bound method combined with an outer approximation method~\cite{duran1986outer}. This implementation was designed to acquire exact optimal solutions, but \texttt{dsOpt} generates a huge number of subproblems in the framework of branch-and-bound, so that computing the solution takes a long time. Hence, there has been a strong desire for  a different approach to solve the EDP in a more practical time.

In contrast to existing implementations~\cite{meuwissen2002gencont, mullin2016using}, this paper is focused on the fact that the crucial quadratic constraint $\x^T \A \x \le 2 \theta$ in (\ref{eq:QCP}) and (\ref{eq:MI-QCP}) can be described as a second-order cone
$\left(
\sqrt{2\theta}N,\U \x \right)\in\mathcal{K}^m$.
The matrix $\U$ is the Cholesky factorization of $\A$ such that $\A=\U^T\U$.
Throughout this paper, we use $\mathcal{K}^m$ to denote the $(m+1)$-dimensional second-order cone:
\begin{equation*}
\mathcal{K}^m = \{(v_0,\v) \in \Real_+ \times \Real^m : ||\v||_2 \le v_0 \}.
\end{equation*}
Introducing a new variable $\y=N\x$, we convert the OCS problem (\ref{eq:MI-QCP}) into an MI-SOCP formulation:
\begin{eqnarray}
\begin{array}{lcl}
\label{eq:MI-SOCP}
\mbox{maximize} &:&  \frac{\g^T\y}{N} \\[0.4em]
\mbox{subject to} &:& \e^T \y = N, \\ 
& & \left(\sqrt{2\theta}N,\U \y \right)\in\mathcal{K}^m, \\
& & y_i \in \{0,1\} \ \mbox{for} \ i=1,\ldots,m.
\end{array} 
\end{eqnarray}
The main difficulty in this MI-SOCP formulation is the non-linearity arising from the second-order cone, and this leads to a heavy computation cost.

In this paper, we propose a lifted polyhedral programming relaxation with active constraint selection method (LPP-ACSM) that removes the non-linearity, exploiting an extension of polyhedral programming relaxation for the second-order cone problem~\cite{ben2001polyhedral, bertsekas1999nonlinear, vielma2008lifted}. We also propose a cone decomposition method (CDM) that is based on cutting-plane methods (geometric cut based on projection~\cite{bialon2003some}) and a Lagrangian multiplier method. In particular, we prove that the Lagrangian multiplier method gives the an analytical solution for orthogonal projection onto the three-dimensional cones, therefore, the proposed CDM generates the linear cuts without relying on iterative methods.

The remainder of this paper is organized as follows. In Section~\ref{sec:2}, we briefly review LPP, then we propose its enhancement LPP-ACSM. Section~\ref{sec:3} gives the details of CDM. The numerical results will be presented in Section~\ref{sec:4}. Finally, in Section~\ref{sec:5}, we formulate some conclusions and discuss for future studies.

\section{Lifted Polyhedral Programming Relaxation}\label{sec:2}

Lifted polyhedral programming (LPP) relaxation~\cite{bertsekas1999nonlinear, vielma2008lifted} is an approach to solve MI-SOCP problems by employing a polyhedral relaxation. Instead of a MI-SOCP problem that involves difficult non-linear constraints, we solve a mixed-integer \textit{linear} programming problem (MI-LP) as the resultant problem. 

Figure~\ref{fig:sub1} illustrates a second-order cone $\KC^m$ with dimension $m=2$. In LPP relaxation, many hyper-planes are generated for constructing a polyhedron cone $\PC_\epsilon^m$ to approximate $\KC^m$ as illustrated in Figure~\ref{fig:sub2}. Here, $\epsilon > 0$ is a parameter to control the tightness of LPP relaxation. Let $\KC_{\epsilon}^m$ be an $\epsilon$ extension of $\KC^m$ defined by  $\mathcal{K}^m_{\epsilon} = \{(v_0,\v) \in \Real_+ \times \Real^m : ||\v||_2 \le (1+\epsilon)v_0 \}$. Ben-Tal and Nemirovski~\cite{ben2001polyhedral} showed that  $\mathcal{P}^m_{\epsilon}$ is wedged between $\KC^m$ and  $\KC_{\epsilon}^m$. More precisely,  $\mathcal{P}^m_{\epsilon}$ satisfies $\mathcal{K}^m \subsetneq \mathcal{P}^m_\epsilon \subsetneq \mathcal{K}^m_{\epsilon}$.

\begin{figure}[htb]
\centering
\begin{subfigure}{.5\textwidth}
  \centering
  \includegraphics[scale=0.52]{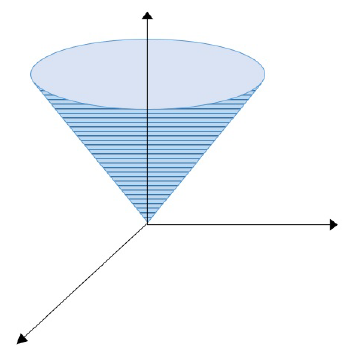}
  \vspace{0.5cm}
  \caption{Second-Order Cone}
  \label{fig:sub1}
\end{subfigure}%
\begin{subfigure}{.5\textwidth}
  \centering
  \includegraphics[scale=0.5]{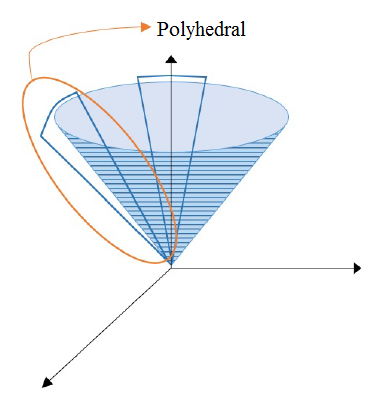}
  \caption{Polyhedral Relaxation}
  \label{fig:sub2}
\end{subfigure}
\caption{Polyhedral Relaxation}
\label{fig:Polyhedral Relaxation}
\end{figure}

Therefore, when we take smaller $\epsilon$, the relaxation $\mathcal{P}^m_\epsilon$ becomes tighter. However, in this case, we require more hyper-planes to build $\mathcal{P}^m_\epsilon$, as will be seen below.

The definition of the polyhedral relaxation $\PC^m_{\epsilon}$ is given in Vielma et al.~\cite{vielma2008lifted}. They first decompose the $m+1$-dimensional second-order cone $\KC^{m}$ into multiple 2-dimensional second-order cone $\KC^2$:
\begin{equation}
\begin{split}
\KC^m := &\{(v_0,\v) \in \Real_+ \times \Real^m : 
\exists (\delta^j)^J_{j=0} \in \Real^{T(m)} \ \mbox{such that} \\[0.3em]
&\qquad v_0 = \delta^J_1, \\[0.3em]
&\qquad \delta^0_i = v_i \ \mbox{for} \ i \in \{1, \cdots, m\},\\[0.3em]
&\qquad \left(\delta^j_{2i-1}, \delta^j_{2i}, \delta^{j+1}_i\right) \in \KC^2 \ \mbox{for} \ i \in \left\{1, \cdots, \left\lfloor \frac{t_j}{2}\right\rfloor \right\},\ j \in\{0,\cdots, J-1\},\\[0.3em]
&\qquad \delta_{t_j}^j = \delta_{\lceil{t_j/2}\rceil}^{j+1} \ \mbox{for} \ j \in \left\{0, \cdots, J-1\right\} \ \mbox{s.t.}  \ t_j \ \mbox{is odd}\}
\end{split}\label{eq:define-Pm}
\end{equation}
with $J = \left\lceil \log_2 m\right\rceil$, and $\{t_j\}^J_{j=0}$ is defined recursively as $t_0 = m$ and $t_{j+1} = \left\lceil\frac{t_j}{2}\right\rceil $ for $j \in \{0, \ldots, J-1\}$
so that $T(m) = \sum_{j=0}^{J} t_j$. For example, $\KC^4$ is determined by a quadratic constraint $v_0^2 \ge v_1^2 + v_2^2 + v_3^2 + v_4^2$. This constraint is decomposed into three constraints $v_0^2 \ge \delta_1 + \delta_2$, $\delta_1 \ge v_1^2 + v_2^2$ and $\delta_2 \ge v_3^2 + v_4^2$, and each of the decomposed constraints can be described as $\KC^2$. 

Then, a replacement of $\KC^2$ in \eqref{eq:define-Pm}
by $\mathcal{W}_j (\epsilon)$ defined below
generates $\PC_{\epsilon}^m$:
\begin{equation}\label{eq:ws}
\begin{split}
\mathcal{W}_j(\epsilon) :=& \left\{(v_0, v_1, v_2) \in \Real_+ \times \Real^2 : \exists(\alpha, \beta) \in \Real^{2{s_{j} (\epsilon)}} \ \mbox{s.t}\right.\\
&\qquad v_0 = \alpha_{s_{j} (\epsilon)} \cos \left(\frac{\pi}{2^{s_{j} (\epsilon)}}\right) + \beta_{s_{j}(\epsilon)} \sin \left(\frac{\pi}{2^s}\right),\ \alpha_1 = v_1 \cos (\pi) + v_2 \sin (\pi),\\
&\qquad \beta_1 \ge |v_2 \cos (\pi) - v_1 \sin (\pi)|,\ \alpha _{i+1} = \alpha_{i} \cos \left(\frac{\pi}{2^i}\right) + \beta_i \sin \left(\frac{\pi}{2^i}\right),\\ 
&\qquad \beta _{i+1} \ge \left| \beta_{i} \cos \left(\frac{\pi}{2^i} \right) - \alpha_i \sin \left(\frac{\pi}{2^i}\right) \right|,\left.\ \mbox{for} \ i \in \{1, \ldots, {s_{j} (\epsilon)}-1\}\right\}
\end{split}
\end{equation}
where
\begin{equation*}
s_{j} (\epsilon) = \ceil *{\frac{j+1}{2}} - \ceil*{\log_4\left(\frac{16}{9}\pi^{-2}\log(1+\epsilon)\right)} \ \mbox{for} \ j\in\{0, \ldots, J-1\}.
\end{equation*}
Note that the number of linear constraints in $\mathcal{W}_j (\epsilon)$ is $5 + 3 (s_j(\epsilon) -1)$, when we divide each linear inequality that involves absolute values into two linear inequalities. 

As a preliminary experiment, we tested this approach across of range of parameter values for $2\theta$ and $\epsilon>0$. To solve the resultant MI-LP, we used the \texttt{CPLEX} package and set the duality gap $5\%$ as the stopping criterion of \texttt{CPLEX}, so the accuracy of the obtained objective value is 5\%. We implemented this approach using \texttt{Matlab R2017b} on a 64-bit Windows 10 PC, Xeon CPU E3-1231 (3.40 GHz) with 8GB memory space. Moreover, we set the number of chosen candidates as $N=50$.

Table~\ref{table:lppresult} shows the numerical results of the LPP relaxation for moderately large problems ($m = 200,1050$ and $2045$), changing $(1+\epsilon)2\theta$. (The bottom part of Table~\ref{table:lppresult} includes the numerical results of LPP-ACSM, which will be explained later.) The first column in the table is the problem size $m$, the second the parameter for the diversity constraint $2\theta$, and the third column a relaxation value $(1+\epsilon)2\theta$. The fourth column shows $\epsilon$ that corresponds to the third column. (More precisely, we adjust $\epsilon$ in the fourth column so that we obtain the values in the third column as $(1+\epsilon)2\theta$.) The fifth, sixth, and seventh columns show the computation time in seconds to build the mathematical model including the construction of $\W_j (\epsilon)$, the computation time to solve the mathematical model by \texttt{CPLEX}, and the total computation time, respectively. The last two columns show the test results: the group coancestry  $\x^T\A\x$ and the objective value $\g^T\x$. In the table, "OOM" indicates that \texttt{CPLEX} ran "out of memory" and could not solve the MI-LP.
\begin{table}[htb]
\caption{Numerical results of LPP and LPP-ACSM for gap$=5\%$}\label{table:lppresult}
\vspace{-1em}
\begin{center}
\makebox[\textwidth][c]{\begin{tabular}{c c c  c r r r r r }
\hline
\multirow{2}{*}{$m$} & \multirow{2}{*}{$2\theta$} &\multirow{2}{*}{$(1+\epsilon)2\theta$} & \multirow{2}{*}{$\epsilon$} & \multicolumn{3}{c}{Time (sec)} & Group & Objective \\
\cline{5-7}
& & & & Builder & Solver & Total & Coancestry & Value\\
\hline
\hline
\multicolumn{9}{c}{LPP}\\
\hline
\multirow{3}{*}{200}& \multirow{3}{*}{0.0334} &  0.03363 & 0.0070 & 3.86 & 10.74 & 14.59 & 0.03380 & 24.96\\
& & 0.03380 & 0.0060 & 3.97 & 10.72 & 14.69 & 0.03380 & 24.96\\
& & 0.03373 & 0.0050 & 6.98 & 19.83 & 26.81 & 0.03340 & 24.83\\
\hline
\multirow{3}{*}{1050}& \multirow{3}{*}{0.0627} & 0.06358 & 0.0070 & 549.40 & 9.18 & 558.58 & 0.06311 & 24.91\\
& & 0.06345 & 0.0060 & 549.89 & 11.84 & 561.74 & 0.06298 & 24.87\\
& & 0.06333 & 0.0050 & 951.22 & 25.31 & 976.54 & 0.06129 & 24.54\\
\hline
\multirow{3}{*}{2045}& \multirow{3}{*}{0.0711} & 0.07209 & 0.0070 & 4702.58 & 256.16 & 4958.74 & 0.07140 & 438.31 \\
& & 0.07196 & 0.0060 & \multicolumn{5}{r}{OOM}\\
& & 0.07181 & 0.0050 &  \multicolumn{5}{r}{OOM}\\
\hline
\hline
\multicolumn{9}{c}{LPP-ACSM}\\
\hline
\multirow{3}{*}{200} & \multirow{3}{*}{0.0334} & 0.03387 & 0.0070 & 3.94 & 3.09 & 7.03 & 0.03360 & 24.94\\
&  & 0.03380 & 0.0060 & 4.09 & 2.72 & 6.82 & 0.03360 & 24.94\\
& & 0.03373 & 0.0050 & 6.57 & 41.18 & 47.75 & 0.03340 & 24.84 \\
\hline
\multirow{3}{*}{1050}& \multirow{3}{*}{0.0627} & 0.06358  & 0.0070 & 568.67 & 34.27 & 602.94  & 0.06307 & 24.81 \\
& & 0.06345 & 0.0060 & 655.75 & 396.26 & 1052.01  & 0.06307 & 24.82\\
& & 0.06333 & 0.0050 & 1012.67 & 217.33 & 1230.01 & 0.06215 & 24.72 \\
\hline
\multirow{3}{*}{2045}& \multirow{3}{*}{0.0711} & 0.07209 & 0.0070 &  4974.45 & 434.32  & 5408.78 & 0.06020  & 429.02 \\
& & 0.07196 & 0.0060 &  &  &  &  & OOM \\
& & 0.07181 & 0.0050 &  &  &  &  & OOM\\
\hline
\end{tabular}}
\end{center}
\end{table}

Through Table~\ref{table:lppresult}, we observe that the group coancestry $\x^T \A \x$ is closer to the threshold $2\theta$ when we use smaller $\epsilon>0$. This confirms that smaller $\epsilon$ leads to a tighter relaxation on the second-order cone. For $m=200$, the problem with $\epsilon=0.0060$ is not optimal since the diversity constraint $\x^T\A\x \le 2\theta$ is violated. Thus, we should use tighter $\epsilon$, but the tighter $\epsilon$ requires longer computation time due to the rapid increment in a number of linear constraints of the LPP relaxation. In addition, for the larger problem $m=2045$ with very tight $(1+\epsilon)2\theta$,  this approach fails to obtain the solution due to out of memory. Therefore, we need an efficient scheme to reduce the large number of linear constraints.

\subsection{LPP relaxation with the active constraint selection method}\label{sec:LPP-ACSM}
We embed an active constraint selection method into the LPP approach to reduce the number of linear inequalities.

\begin{DEFI}[Active Constraint] Let $\a_i^T\x \le b_i \ (i=1,\ldots,p)$ be inequality constraints
in an optimal optimization problem $P$ with $\a_i\in\mathbb{R}^{q}$ 
and $b_i\in\mathbb{R}$ \ $(i = 1,\ldots, p)$, 
and let $\x^*$ be an optimal solution of the optimization problem $P$.
The inequality constraint 
$\a_i^T\x \le b_i$ is said to be {\upshape active} at $\x^*$
if $\a_i^T\x^*=b_i$.
Otherwise, the constraint $\a_i^T\x \le b_i$ is called {\upshape inactive} at $\x^*$.
\end{DEFI}
\begin{ALGO}[Active constraint selection method]
Let $P(\epsilon)$ be an optimal solution that is defined with a parameter $\epsilon$ and $\x^*(\epsilon)$ be its optimal solution. Let $S = \{\epsilon_1, \epsilon_2, \ldots, \epsilon_l \}$ be a parameter set $\epsilon$ in preliminary experiments. If an inequality constraint $\a_i^T x \le b_i$ in $P(\epsilon)$ is active at $\x^*(\epsilon)$ for any $\epsilon \in S$, 
we replace $\a_i^T \x \le b_i$ with the equality constraint $\a_i^T \x = b_i$.
\end{ALGO}
Using such method, we conducted preliminary experiments with the parameter set $S = \{0.04,$  $0.05,0.08\}$ (different from $\epsilon$ in Table~\ref{table:lppresult}) and found that the constraint
\begin{equation*}
\beta_1 \ge -v_2 \cos (\pi) +v_1 \sin (\pi)
\end{equation*}
in (\ref{eq:ws}) was always active at $\x^*(\epsilon)$ for all $\epsilon \in S$. Therefore, we replace the constraint $\beta_1 \ge v_2 \cos (\pi) -v_1 \sin (\pi)$ with the equality $\beta_1 = -v_2 \cos (\pi) +v_1 \sin (\pi)$. This replacement can reduce the number of inequalities, and we could expect the reduction in computation time.

The latter half of Table~\ref{table:lppresult} shows the computation time in the framework of LPP relaxation combining with the active constraint selection method (LPP-ACSM). The same parameter as LPP was set for LPP-ACSM. We observed that the group coancestry $\x^T \A \x$ and the objective values $\g^T \x$ are similar for LPP and LPP-ACSM, but contrary to our expectation, the computation time for LPP-ACSM was longer.

Table~\ref{table:nnz} compares the number of rows, columns, and nonzero elements in MI-LP problems that are solved by LPP and LPP-ACSM for $\epsilon=0.007$.
We notice that LPP-ACSM generated more number of rows, columns, and nonzero elements than did LPP. In LPP-ACSM, CPLEX tried to remove $\beta_1$ by substituting $\beta_1$ with $-v_2 \cos (\pi) +v_1 \sin (\pi)$ through its preprocessing phase, but this sacrificed the sparsity in LPP. Thus, LPP-ACSM incurs more nonzero elements than LPP, and this makes LPP-ACSM slow to obtain the result.

\begin{table}[htb]
	\begin{center}
		\caption{The number of nonzero elements in MI-LP problems arising from LPP and LPP-ACSM}\label{table:nnz}
		\vspace{-1em}
		\makebox[\textwidth][c]
		{\begin{tabular}{c | r r r | r r r}
			\hline
			\multirow{2}{*}{$m$}  & \multicolumn{3}{c|}{LPP} & \multicolumn{3}{c}{LPP-ACSM}\\
			\cline{2-7}
			&  \multicolumn{1}{c}{$\#$ row} & \multicolumn{1}{c}{$\#$ column} & \multicolumn{1}{c|}{$\#$ nonzeros} & \multicolumn{1}{c}{$\#$ row} & \multicolumn{1}{c}{$\#$ column} & \multicolumn{1}{c}{$\#$ nonzeros}  \\
			\hline
			200 & 1572 & 1059 & 6725 & 1497 & 984 & 6575 \\
			1050 & 8303 & 5806 & 52611 & 60100 & 33476 & 355316 \\
			2045 & 16175 & 11364 & 71305 & 116896 & 65073 & 660721\\
			\hline 
		\end{tabular}}
	\end{center}
\end{table}

The main disadvantage of LPP was the large number of linear constraints. This disadvantage is more critical when the relaxation is tight, and cannot be completely removed by LPP-ACSM. In the next section, therefore, we propose another approach for solving EDP~\eqref{eq:MI-SOCP}.

\section{Cone Decomposition Method}\label{sec:3}
In this section, we propose a cone decomposition method (CDM) for EDP~\eqref{eq:MI-SOCP}. The basic concept of the cone decomposition method also draws on the properties of second-order cones. The above LPP approach decomposes an $m+1$-dimensional second-order cone $\KC^m$ into multiple two-dimensional second-order cones $\KC^2$ in a recursive style as shown in  \eqref{eq:define-Pm}. In constrast, the proposed CDM makes use of different decomposition, based on the following theorem from~\cite{vielma2017extended}.
\begin{THEO}\label{cor1}
{\upshape \cite{vielma2017extended}}
Let
\begin{equation*}
\hat{\H}^m:=\left\{(v_0,\v,\w)\in\Real^{(2m+1)}:v_j^2\leq w_j v_0,\forall j\in\{1, \ldots, m\},\sum_{j=1}^{m }w_j\leq v_0\right\},
\end{equation*}
then $\mathcal{K}^m=\mbox{Proj}_{(v_0,\v)}(\hat{\H}^d)$, where $\mbox{Proj}_{(v_0,\v)}$ is the orthogonal projection onto the space of $(v_0,\v)$ variables.
\end{THEO}
Theorem \ref{cor1} gives another decomposition of $\KC^m$ by using an auxiliary vector $\w \in \Real^m$.
\begin{CORO}\label{coro1}
	A second-order cone $\mathcal{K}^m$ can be also written as 
	\begin{equation*}
	\mathcal{K}^m:=\left\{(v_0,\v)\in\Real^{(m+1)}:\exists \w \in \Real^m \ \textrm{such that} \ \v_j^2\leq w_j v_0,\forall j\in\{1,\ldots,m\},\sum_{j=1}^{m}w_j\leq v_0\right\}.
	\end{equation*}
\end{CORO}
The utilization of Corollary \ref{coro1} leads to another reformulation of our OCS (3) as follows:
\begin{equation}\label{eq:cdm}
\begin{array}{ll}
\mbox{maximize} & :  \frac{\g^T\y}{N} \\[0.3em]
\mbox{subject to} & :  \e^T\y=N,\\
& \z= \U^T \y \ \text{for}\ i=1,\ldots,m \\
& z_i^2 \le w_i c_0 \ \text{for}\ i=1,\ldots,m ,\\[0.2em]
&  \sum_{i=1}^{m} w_i \le c_0,\\[0.2em]
&  y_i\in\{0,1\} \ \text{for}\ i=1,\ldots,m
\end{array}
\end{equation} 
where $z_i$ is the $i$th element of $\z$ and $c_0 = \sqrt{2\theta N^2}$. In this new formulation, the decision variables are $\y, \z$, and $\w$. 

The nonlinear constraint in \eqref{eq:cdm} is only the quadratic constraint $z_i^2 \le w_i c_0$. In the proposed CDM, we generate the cutting planes to these quadratic cones.
The framework of the proposed CDM is given as Algorithm~\ref{algo:cdm}.

\begin{ALGO}\label{algo:cdm} {\upshape A framework for the cone decomposition method.
\begin{itemize}
    \item [Step 1]
    Set a threshold $\delta \ge 0$, for example $\delta = 10^{-8}$.
    \item [Step 2]
    Let $P^0$ be an MI-LP problem that is generated from an optimization problem (\ref{eq:cdm}) by omitting the quadratic constraints $z_i^2 \le w_i c_0 \ (i=1, \ldots, m)$.
    Apply an MI-LP solver to $P^0$, and let its optimal solution be
     $\left(\hat{\y}^0, \hat{\z}^0, \hat{\w}^0\right)$. Let $k=0$.
     \item [Step 3]
     Let a set of generated cuts $\CC^k = \emptyset$.
 \item[Step 4]
For each $i = 1, \ldots, m$, if $(\hat{z}_i^k)^2 \le \hat{w}_i^k c_0$ is violated, apply the following steps.
\begin{itemize}
\item[Step 4-1] Compute the orthogonal projection of $(\hat{z}_i^k, \hat{w}_i^k)$ onto $z_i^2 \le w_i c_0$ by solving the following subproblem with the Lagrangian multiplier method.
\begin{equation*}
\begin{array}{ll}
\mbox{minimize} &: \frac{1}{2}\left(\bar{z}-\hat{z}_i^k\right)^2+\frac{1}{2}\left(\bar{w}-\hat{w}_i^k\right)^2\\
\mbox{subject to} &: \bar{z}^2\le \bar{w} c_0
\end{array}
\end{equation*}
Let $(\bar{z}_i^k,\bar{w}_i^k)$ be the solution of this subproblem.
\item[Step 4-2] 
Add to $\CC^k$ the following linear constraint 
\begin{equation*}
\left(\begin{array}{c}
\hat{z}_i^k - \bar{z}_i^k \\ \hat{w}_i^k - \bar{w}_i^k 
\end{array} \right)^T
\left(\begin{array}{c}
z_i - \bar{z}_i^k \\ w_i - \bar{w}_i^k 
\end{array} \right) \le 0.
\end{equation*}
\end{itemize}
\item [Step 5]
If $\CC^k$ is empty, output $\hat{\y}^k$ as the solution and terminate.
\item [Step 6]
Build a new MI-LP $P^{k+1}$ by adding $\CC^k$ to $P^k$.
Let the optimal solution of $P^{k+1}$ be $\left(\hat{\y}^{k+1}, \hat{\z}^{k+1}, \hat{\w}^{k+1}\right)$.
If 
\begin{equation}
||\hat{\z}^{k+1}-\hat{\z}^k|| \le \delta \ \mbox{and} \ ||\hat{\w}^{k+1}-\hat{\w}^k|| \le \delta,
\end{equation}
output $\hat{\y}^k$ as the solution and terminate.
\item [Step 7]
Return to Step 3 with $k \leftarrow k+1$.
\end{itemize}
} 
\end{ALGO}

In Step~4-1 of Algorithm~\ref{algo:cdm}, we compute the orthogonal projection. It would be desirable to compute the orthogonal projection on the original quadratic constraint $\x^T \A \x \le 2 \theta$, such orthogonal projection does not have an analytic form. Kiseliov~\cite{kiseliov1994algorithms} proposed
some numerical method, but this is an iterative method. Another iterative method is also proposed by~\cite{ferreira2018project} to solve different case of second-order cones. In contrast, the orthogonal projection in Step~4-1 is onto the quadratic constraint $\bar{z}^2 \le \bar{w} c_0$. We can derive the analytical form, as proven in the next theorem. Note that the decomposition in \eqref{eq:cdm} enables us to derive this theorem.
\begin{THEO}\label{theorem:cdm}
Assume that $(\hat{z}, \hat{w}) \in \Real^2$ satisfies $\hat{z}^2 > \hat{w} c_0$. Let $(\bar{z}, \bar{w}) \in \Real^2$ be the orthogonal projection of $(\hat{z}, \hat{w})$ onto $z^2 \le w c_0$. Then, $(\bar{z}, \bar{w})$ can be given by an analytical form.
\end{THEO}

In the proof of Theorem \ref{theorem:cdm}, we make use of Cardano's  Formula \cite{witula2010cardano} to obtain a root of a cubic function analytically.
\begin{THEO}\label{theorem:cardano}{\upshape [Cardano's Formula \cite{witula2010cardano}]}
	{Let $F(\lambda)$ be a cubic function $ F(\lambda) = a\lambda^3+b\lambda^2+c\lambda+d$ with $a\ne0$. Then $F(\lambda) = 0$ has three solutions}
	\begin{eqnarray*}
	\left\{\begin{array}{lcl}
	\lambda_1 &=& S+T-\frac{b}{3a} \\
	\lambda_2 &=& -\frac{S+T}{2}-\frac{b}{3a}+\frac{i\sqrt{3}}{2}(S-T)\\
	\lambda_3 &=& -\frac{S+T}{2}-\frac{b}{3a}-\frac{i\sqrt{3}}{2}(S-T),
	\end{array}\right.
	\end{eqnarray*}
	where
	\begin{eqnarray*}
	S = \sqrt[3]{R+\sqrt{Q^3+R^2}}, \quad T = \sqrt[3]{R-\sqrt{Q^3+R^2}}, \quad
	Q =\frac{3ac-b^2}{9a^2}, \ \mbox{\upshape and} \   R=\frac{9abc-27a^2d-2b^3}{54a^3}.
	\end{eqnarray*}
\end{THEO}

\begin{proof} (for Theorem~\ref{theorem:cdm})

The orthogonal projection $(\bar{z}, \bar{w}) \in \Real^2$ is the optimal solution of the following subproblem.
\begin{equation}
\begin{array}{ll}
\mbox{minimize} &: \frac{1}{2}\left(z-\hat{z} \right)^2+\frac{1}{2}\left(w -\hat{w} \right)^2\\
\mbox{subject to} &: z^2\le w c_0.
\end{array}
\label{eq:theo}
\end{equation}
This problem has a convex closed feasible region and its objective function is strongly convex, therefore, this problem has a unique solution.
Since $(\hat{z}, \hat{w})$ is outside of the region $z^2 \le w c_0$,
the projection exists on the boundary of the region. We can replace $z^2\le w c_0$ with $z^2 = w c_0$, and \eqref{eq:theo} is equivalent to the following optimization problem:
\begin{equation}
\begin{array}{ll}
\mbox{minimize} &: \frac{1}{2}\left(z-\hat{z} \right)^2+\frac{1}{2}\left(w -\hat{w} \right)^2\\
\mbox{subject to} &: z^2 = w c_0
\end{array}
\label{eq:theo2}
\end{equation}

To apply a Lagrangian multiplier method, we prepare a Lagrangian function of \eqref{eq:theo2} with a Lagrangian multiplier $\lambda \in \Real$:
\begin{equation*}
\mathcal{L}(z, w, \lambda)=\frac{1}{2}(z-\hat{z})^2+\frac{1}{2}(w-\hat{w})^2-\lambda(w c_0-z^2).
\end{equation*}
Setting $\nabla\mathcal{L}=0$, we have
\begin{align}
\nabla_{z}\mathcal{L} &= z-\hat{z}+2\lambda c_0 = 0,\label{eq:4} \\
\nabla_{w}\mathcal{L} &= w-\hat{w}-\lambda c_0 = 0,\label{eq:5} \\
\nabla_{\lambda}\mathcal{L} &= -c_0w+z^2=0. \label{eq:6} 
\end{align}
Substituting \eqref{eq:4} and \eqref{eq:5} into \eqref{eq:6} leads to a cubic function with respect to $\lambda$:
\begin{equation}
4c_0^2\lambda^3+(4c_0^2+4c_0\hat{w})\lambda^2+(c_0^2+4c_0\hat{w})\lambda+(c_0\hat{w}-(\hat{z})^2)=0 \label{eq:7}
\end{equation}

Defining $a = 4c_0^2,\ b= 4c_0^2+4z_0\hat{w},\ c=c_0^2+4z_0\hat{w},\text{and}\ d=c_0\hat{w}-\hat{z}^2$, we apply Theorem~\ref{theorem:cardano} to obtain $\lambda$. In Theorem~\ref{theorem:cardano}, we have three solutions $\lambda_1, \lambda_2, \lambda_3$. Among three solutions, only $\lambda_1$ can generate the analytical solution since the other two are complex numbers. 
To prove that $\lambda_2$ and $\lambda_3$ are complex numbers, it is enough to show $S \neq T$, and this is equivalent to show $Q^3+R^2\neq0$.
Computing
\begin{align*}
Q^3+R^2 &= \left(\frac{3ac-b^2}{9a^2}\right)^3+\left(\frac{9abc-27a^2d-2b^3}{54a^3} \right)^2\\[1em]
&= \frac{(3ac-b^2)^3}{729a^6}+ \frac{(-27 a^2 d + 9 a b c - 2 b^3)^2}{2916 a^6}\\[1em]
&= \frac{27 a^2 d^2 - 18 a b c d + 4 a c^3 + 4 b^3 d - b^2 c^2}{108 a^4},
\end{align*}
we substitute $a,b,c,d$ and $(\hat{z})^2\leq(\hat{w})^2$. Therefore, for $c_0=\sqrt{2\theta N^2}\ne0$ and $\hat{z},\hat{w}\in\{\Real\}\backslash\{\0\}$,  $\lambda_2$ and $\lambda_3$ are complex numbers. 

Thus, we only have the analytical solution by $\lambda_1$. After we obtain $\lambda$ as $\lambda_1$, it is easy to compute $z$ and $w$ by \eqref{eq:4} and \eqref{eq:5}. Therefore, the optimal solution $(\bar{z}, \bar{w})$ of \eqref{eq:theo} has a analtycial form.
\end{proof}

The termination of the proposed method is guaranteed by the following theorem.
\begin{THEO}
Algorithm \ref{algo:cdm} terminates in a finite number of iterations.
\end{THEO}
\begin{proof}
The number of points we are interested for $\y$ is at most $2^m$,
where $m$ is the number of candidate genotypes,
due to the binary constraints $y_i \in \{0, 1\}$.
In $k$ iterations, the generated cuts in $\CC^k$ remove $\hat{\z}^k, \hat{\w}^k$.
Since $\hat{\y}^k$ is directly connected to $\hat{\z}^k$ by the constraint
$\z = \U^T \y$ and $\U$ is invertible, $\hat{\y}^k$ is not feasible in $P^{k+1}$.
At least one solution will be infeasible in each iteration, therefore, the number of iterations is also at most $2^m$.
\end{proof}
Since any feasible point is not excluded by the generated cuts,
we can find an optimal solution if the stopping threshold is $\delta = 0$.

\section{Numerical Results}\label{sec:4}
Numerical experiments were conducted to compare the performance of the proposed methods (LPP-ACSM and CDM) with existing software (\texttt{dsOpt} as implemented in \texttt{OPSEL}) and \texttt{GENCONT}, a general MI-SOCP solver \texttt{CPLEX}, and LPP itself. The proposed methods were implemented using Matlab R2017b by setting \texttt{CPLEX} as the solver of MI-LP. All methods were executed on a 64-bit Windows 10 PC with Xeon CPU E3-1231 (3.40 GHz) and 8 GB memory space. The data were taken from \url{https://doi.org/10.5061/dryad.9pn5m} or generated by the simulation POPSIM~\cite{mullin2010using}. The sizes of the test instances are $m$=200, 1050, 2045, 5050, 10100, and 15222. We set parameter $N=50,100$, and as a stopping criterion for \texttt{CPLEX}, we used $gap=1\%, 5\%$ . We also chose $\delta = 10^{-8}$. The computation time was limited to 3 hours for all methods.

First, Table~\ref{result:GENCONT} shows the results from the OCS solver \texttt{GENCONT}. In this table, the columns ``$\g^T \x$" and ``$\x^T \A \x$" are the obtained objective values and group coancestry, respectively. We only show the solution for $m \leq 5050$, since the results with $m=10100, 15222$ cannot be obtained due to out of memory. From Table \ref{result:GENCONT}, we observe that the number of chosen candidates did not match the given parameter $N$. This indicates that \texttt{GENCONT} failed to output feasible solutions.

\begin{table}[htb]
\begin{center}
\caption{Numerical results on 
\texttt{GENCONT}}\label{result:GENCONT}
{\small
\begin{tabular}{r r r r r r }
\hline
\multicolumn{6}{c}{$N=50$}\\
\hline
\multicolumn{1}{c}{$m$} & \multicolumn{1}{c}{$2\theta$} & \multicolumn{1}{c}{$\g^T\x$} & \multicolumn{1}{c}{$\x^T\A\x$} & \multicolumn{1}{c}{time (sec)} & \multicolumn{1}{c}{$\#$ selected $N$}\\
\hline
200 & 0.0334 & 11.472 & 0.03340 & 3.54 & 64\\ 
1050 & 0.0627 & 25.91 & 0.06270 & 7.20 & 81\\
2045 & 0.0711 & 438.36 & 0.07109 & 111.52 & 71\\
5050 & 0.1081 & 43.44 & 0.10810 & 1561.43 & 78\\
\hline
\multicolumn{6}{c}{$N=100$}\\
\hline
\multicolumn{1}{c}{$m$} & \multicolumn{1}{c}{$2\theta$} & \multicolumn{1}{c}{$\g^T\x$} & \multicolumn{1}{c}{$\x^T\A\x$} & \multicolumn{1}{c}{time (sec)} & \multicolumn{1}{c}{$\#$ selected $N$}\\
\hline
200 & 0.0258 & 8.89 & 0.02580 & 0.48& 93\\
1050 & 0.0539 & 24.07& 0.0539 & 4.77 & 94\\
2045 & 0.0628 & 432.75 & 0.06279 & 106.48& 74\\
5050 & 0.0994 & 42.08 & 0.09940 & 1533.31 & 81\\
\hline
\end{tabular}
}
\end{center}
\end{table}

The results for the other methods where $N = 50$  are presented in Table~\ref{table:result50}. For the LPP relaxation and its modification (LPP-ACSM), we fixed $\epsilon=0.005$ so that these two methods output feasible solutions. In addition, since only LPP and LPP-ASCM require the parameter $\epsilon$, we show the value of $\epsilon$ for only two methods in the column $(1+\epsilon)2\theta$. The other methods \texttt{CPLEX}, \texttt{dsOpt} and CDM do not need the parameter $\epsilon$,
and this is indicated by ``*" in the table. When the computation could not finish the computation within the time limit of 3 hours, it is indicated as '$>$ 3 hours' and the best objective values up to that point are shown in the table.

From Table \ref{table:result50}, LPP and LPP-ACSM failed to obtain the solution due to OOM (out of memory) for large problems $m \geq 2045$. To attain the feasibility, we set a relatively small $\epsilon$, but this demanded a huge number of linear constraints, as discussed in Section~\ref{sec:2}.

\begin{table}[htb]
\begin{center}
\caption{Numerical comparison for EDPs  ($N = 50$)}\label{table:result50}
{\small 
\begin{tabular}{l|c|c|c|r|c|r|r|c|r}
\hline
\multirow{2}{*}{Algorithm}& \multirow{2}{*}{$m$} & \multirow{2}{*}{$2 \theta$} & \multirow{2}{*}{$(1+\epsilon)2\theta$}& \multicolumn{3}{c|}{gap $=5\%$}& \multicolumn{3}{c}{gap $=1\%$}\\
\cline{5-7}\cline{8-10}
& & & & $\g^T \x$ & $\x^T \A \x$  &  time (sec) & $\g^T \x$ & $\x^T \A \x$  &  time (sec) \\
\hline

\texttt{CPLEX} & \multirow{5}{*}{200} & \multirow{5}{*}{0.0334} & * & 24.99 & 0.03340 & 1.06 & 25.19 & 0.03340 & 8735.24\\
\texttt{dsOpt} & &  & * &  25.12 & 0.03340 & 5.32  & 25.18 & 0.03340 & 606.94 \\
LPP &  &  & 0.03373 & 24.83 & 0.03340 & 26.81 & 25.11 & 0.03340 &  3691.26 \\
LPP-ACSM &  &  & 0.03373 & 24.84 & 0.03340 & 47.75 & 25.15 & 0.03340 &   2587.16\\
CDM & &  & * &  25.02 &  0.03340 &  2.37 & 25.15 & 0.03340 & 9.89\\
\hline

\texttt{CPLEX} & \multirow{5}{*}{1050} & \multirow{5}{*}{0.0627} & * &  24.97 & 0.06267 & 3.56& 24.97 & 0.06267 & 6.64\\
\texttt{dsOpt} & &  &* &  24.97 & 0.06169 & 5.19 & 24.85 & 0.06268 & $>$ 3 hours \\ 
LPP &  &  & 0.06333 & 24.54 & 0.06129 & 976.54 & 24.89 & 0.06291 &   10063.39 \\
LPP-ACSM &  &  & 0.06333 & 24.72 & 0.06215 & 1230.01 & 24.89 & 0.06291 & 1634.58  \\
CDM &  &  & * &  24.65 & 0.06118 &  8.67 & 24.96 & 0.06238 & 12.83\\
\hline

\texttt{CPLEX} & \multirow{5}{*}{2045} & \multirow{5}{*}{0.0711} & * & 437.21 & 0.07100 & 3.95& 437.21 & 0.07100 & 3.83\\
\texttt{dsOpt} &  &  & * & 432.94 & 0.06700 & 7.09& 435.87 & 0.07020 & 14.42\\
LPP &  &  & 0.07181 & &  & OOM & &&   OOM\\
LPP-ACSM &  &  & 0.07181 &  &  & OOM &  &  & OOM\\
CDM &  &  & * &  434.26 & 0.06760 & 1.80 & 437.38 & 0.06960 & 2.61\\
\hline

\texttt{CPLEX} & \multirow{5}{*}{5050} & \multirow{5}{*}{0.1081} & * & 41.90 & 0.10776 & 73.16& 42.57 & 0.10781 & $>$ 3 hours \\ 
\texttt{dsOpt} &  &  & * &  41.57 & 0.10471 & 236.70& 42.67 & 0.10807 & $>$ 3 hours \\ 
LPP &  &  & 0.109184 & &  & OOM & &&   OOM\\
LPP-ACSM &  &  & 0.109184 & &  & OOM & &&   OOM\\
CDM &  &  & * & 42.56  & 0.10742 & 171.85 & 42.56 & 0.10742 & 179.05 \\
\hline

\texttt{CPLEX} & \multirow{5}{*}{10100} & \multirow{5}{*}{0.0701}  & * & 44.89 & 0.06931 & $>$ 3 hours 
& 44.89 & 0.06931 & $>$ 3 hours \\ 
\texttt{dsOpt} &  &  & * &  46.00 & 0.07005 & 4509.83& 46.21 & 0.06975 & 8787.37 \\
LPP &  &  & 0.070803 & &  & OOM & &&   OOM\\
LPP-ACSM &  &  & 0.070803 & &  & OOM & &&   OOM\\
CDM &  &  & * &  45.27 & 0.06896 & 1131.14 & 46.43 & 0.07005 & 1431.12 \\
\hline

\texttt{CPLEX} & \multirow{5}{*}{15222} & \multirow{5}{*}{0.0388} & * &  118.33 & 0.03840 & $>$ 3 hours 
& 107.56 & 0.03280 & $>$ 3 hours \\ 
\texttt{dsOpt} &  &  & * &   &  & OOM &  &  & OOM \\
LPP &  &  & 0.039189 & &  & OOM & &&   OOM\\
LPP-ACSM &  &  & 0.039189 & &  & OOM & &&   OOM\\
CDM &  &  & * &  452.57 & 0.03880 & 493.85 & 461.83 & 0.03880 & 1111.49\\
\hline
\end{tabular}
} 
\end{center}
\end{table}

In contrast to LPP and LPP-ACSM, \texttt{CPLEX} shows its computation efficiency when gap $=5\%$. However, for larger problems or smaller gaps, \texttt{CPLEX} consumes more time than other methods. For example, we can see a large time difference for the smallest size $m=200$. \texttt{CPLEX} for gap$=5\%$ is the most efficient method among the five methods, but it becomes the slowest method when we set the gap as $1\%$. For such a tight gap, our proposed approach CDM can reduce computation time to less than $10$ seconds. In addition, for $m = 15222$, \texttt{CPLEX} could not finish its computation within the time limit (3 hours), and the best objective value at 3 hours was much worse than CDM; while CDM obtained $\g^T\x=452.57$, \texttt{CPLEX} only reached $\g^T\x=118.33$.

Table \ref{tab:result100} shows the results for all methods, except \texttt{GENCONT}, for the case of $N=100$. Similar to the results in Table~\ref{table:result50}, LPP and LPP-ACSM failed to solve large problems due to insufficient memory. From both tables, \texttt{dsOpt} actually gives similar performance to \texttt{CPLEX}. However, for the largest problem ($m = 15222$), \texttt{dsOpt} failed with "out of memory". In contrast to the other methods, CDM obtains feasible solutions without having a memory problem. Thus, CDM not only reduces the computation time, but also the memory usage. Based on these observation, CDM is the most effective method to solve OCS problem.

\begin{table}[htb]
\begin{center}
\caption{Numerical comparison for EDPs ($N = 100$)}\label{tab:result100}
{\small 
\begin{tabular}{l|r|c|c|r|c|r|r|c|r}
\hline
\multirow{2}{*}{Algorithm}& \multirow{2}{*}{$m$} & \multirow{2}{*}{$2 \theta$} & \multirow{2}{*}{$(1+\epsilon)2\theta$}& \multicolumn{3}{c|}{gap $=5\%$}& \multicolumn{3}{c}{gap $=1\%$}\\
\cline{5-7}\cline{8-10}
& & & & $\g^T \x$ & $\x^T \A \x$  &  time (sec) & $\g^T \x$ & $\x^T \A \x$  &  time (sec) \\
\hline

\texttt{CPLEX} & \multirow{5}{*}{200} & \multirow{5}{*}{0.0258} & * & 23.19  & 0.02580 & 4.31  & 23.49 & 0.02580 & 13.14\\
\texttt{dsOpt} & &  & * &  23.14 & 0.02575 & 1.30 & 23.54 & 0.02580 & 566.89\\
LPP &  &  & 0.026059 &  23.38 & 0.02585 & 10.60  & 23.53 & 0.02585 & 2289.65\\
LPP-ACSM &  &  & 0.026059 & 23.24 & 0.02580 &  11.35 & 23.58 & 0.02585 & 4003.32\\
CDM & &  & * &  23.53 &  0.02580 &  1.63 & 23.55 & 0.02580 & 1.89\\
\hline

\texttt{CPLEX} & \multirow{5}{*}{1050} & \multirow{5}{*}{0.0539} & * & 22.53 & 0.05389 & 6.68  & 22.53 & 0.05389 & 3.64\\
\texttt{dsOpt} & &  & * & 21.79 & 0.05358 & 6.07& 22.25 & 0.05382 & 193.08\\
LPP &  &  & 0.054440 & 22.35 & 0.05401 & 1047.81 & 22.35 & 0.05401 & 1088.16  \\
LPP-ACSM &  &  & 0.054440 & 22.34 & 0.05392  & 1019.46   &  &  &  OOM\\
CDM &  &  & * &  22.49 & 0.05339 &  14.78& 22.49 & 0.05339 & 14.51\\
\hline

\texttt{CPLEX} & \multirow{5}{*}{2045} & \multirow{5}{*}{0.0628} & * & 420.04  & 0.06100  & 3.21 & 420.04 & 0.06100 & 3.08\\
\texttt{dsOpt} &  &  & * &  419.53 & 0.06155 & 7.93& 419.53 & 0.06155 & 7.96\\
LPP &  &  & 0.063429 & &  &   OOM & & &  OOM \\
LPP-ACSM &  &  & 0.063429 &   &  &  OOM&  &  &  OOM\\
CDM &  &  & * & 418.67 & 0.06010 & 2.54 &  418.67 & 0.06010 & 2.42\\
\hline

\texttt{CPLEX} & \multirow{5}{*}{5050} & \multirow{5}{*}{0.0994} & * & 40.63 & 0.09932 & 58.37  & 40.63 & 0.09932 & 54.43    \\
\texttt{dsOpt} &  &  & * & 40.13 & 0.09860 & 134.55& 40.47 & 0.09936 & 367.29\\
LPP &  &  & 0.100498 & &  &   OOM & & &  OOM \\
LPP-ACSM  &  &  & 0.100498 & & &   OOM & & &  OOM \\
CDM &  &  & * & 40.28  & 0.09821 & 168.21 &  40.35 & 0.09742 & 183.25 \\
\hline

\texttt{CPLEX} & \multirow{5}{*}{10100} & \multirow{5}{*}{0.0610} & * & 43.79  & 0.06059 & 2720.18 & 44.34 & 0.06070  & $>$ 3 hours \\ 
\texttt{dsOpt} &  &  & * &  43.36 & 0.06018 & 584.77& 44.44 & 0.06100 & 7538.99\\
LPP &  &  & 0.061611 & & &   OOM & &  & OOM \\
LPP-ACSM  &  &  & 0.061611 & & &   OOM & & &  OOM \\
CDM &  &  & * &  43.86 & 0.06095 & 1003.68 & 44.53 & 0.06092 & 1269.18 \\
\hline
\texttt{CPLEX} & \multirow{5}{*}{15222} & \multirow{5}{*}{0.0300} & * & 436.92  & 0.02990  & 5084.69 & 436.92 & 0.02990 & $>$ 3 hours \\ 
\texttt{dsOpt} &  &  & * &   &  & OOM &  &  & OOM\\
LPP &  &  & 0.030301 & &  &   OOM & & &  OOM \\
LPP-ACSM  &  &  & 0.030301 & &  &   OOM & & &  OOM \\
CDM &   & &*& 432.13 & 0.02865 & 654.64  & 439.88 & 0.02960 & 641.76\\
\hline
\end{tabular}
} 
\end{center}
\end{table}

Finally, Table~\ref{tab:iter-cdm} presents additional evidence supporting the efficiency of the CDM. The column ``\# iteration" is the number of main iterations to obtain the output in the CDM, and the column ``\# constraint" shows the numbers of constraints of MI-LP problems in the first iteration and the last iteration. For example, $76$ and $2357$ in the column ``\# constraint" indicate that the MI-LP problems have 76 rows and 2357 rows in the first iteration and the last iteration, respectively. (See the problem with $N=50$ and $m=200$). This implies that $2357-76=2281$ constraints are added during the CDM. Compared to the rows in LPP and LPP-ACSM of Table~\ref{table:nnz}, the CDM requires considerably fewer MI-LP problems, and this makes the CDM much faster.
From the viewpoint of memory consumption, when we applied the CDM to the problem with $m=15.222$, $N=50$, and gap=5\%, the first MI-LP required 5.4 GB memory space, but the last MI-LP consumed only 2.4 GB memory space. During the CDM iterations, we add linear constraints to MI-LP problems, such that \texttt{CPLEX} can effectively exploit the solution obtained in the previous iteration to find the next solution. Therefore, the first MI-LP requires the greatest memory space, but after the first MI-LP, the CDM demands less memory.

\begin{table}[htb]
\begin{center}
\caption{The number of iterations and constraints in the CDM} \label{tab:iter-cdm}
{\small 
\begin{tabular}{c|r|r r|r|r r|r}
\hline
\multirow{3}{*}{$N$} & \multirow{3}{*}{$m$} & \multicolumn{3}{c|}{gap$=5\%$} & \multicolumn{3}{c}{gap$=1\%$}\\
\cline{3-8}
& & \multirow{2}{*}{\# iteration} & \multicolumn{2}{c}{\# constraint} & \multirow{2}{*}{\# iteration} & \multicolumn{2}{c}{\# constraint}\\
\cline{4-5}\cline{7-8}
& & & First Iter & last Iter & & First Iter & Last Iter\\ 
\hline
\multirow{6}{*}{50} & 200 & 15 & 76 & 2357 & 15 & 76 & 2384\\
& 1050 & 7 & 201 & 1802 & 8 & 201 & 2086\\
& 2045 & 7 & 53 & 789 & 8 & 53 & 925\\
& 5050 & 8 & 129 & 1855 & 8 & 129 & 1855\\
& 10100 & 9 & 189 & 3338 & 12 & 189 & 4671\\
& 15222 & 13 & 56 & 3108 & 38 & 56 & 11449\\
\hline
\multirow{6}{*}{100} & 200 & 9 & 144 & 1654 & 10 & 144 & 1858\\
& 1050 & 10 & 255 & 3337 & 10 & 255 & 3337\\
& 2045 & 8 & 105 & 1696& 8 & 105 & 1696\\
& 5050 & 6 & 200 & 1973 & 6 & 200 & 1972\\
& 10100 & 6 & 281 &  2672 & 8 & 281 & 3733\\
& 15222 & 9 & 106 & 2935 & 10 & 106 & 3300\\
\hline
\end{tabular}
} 
\end{center}
\end{table}

\section{Conclusion and Future Work}\label{sec:5}
In this paper, we proposed LPP with an active constraint selection method (LPP-ACSM) and cone decomposition method (CDM) to achieve optimal contribution selection in the context of tree breeding. We compared the efficiency of the proposed methods with those found in existing breeding selection software (\texttt{GENCONT} and \texttt{dsOpt}), the optimization solver \texttt{CPLEX}, and LPP.
From the numerical results, we observed that LPP and LPP-ACSM failed to obtain solutions for problems with large $m$ due to insufficient memory. Since we used very tight $\epsilon$, the number of constraints was huge. 


Our final proposed method, CDM, can efficiently obtain the optimal solution of EDP problems. For the largest problem $m = 15222$, while \texttt{CPLEX} could not find satisfactory solutions in 3 hours, the CDM can still efficiently obtain a feasible solution without having memory limitations. The use of CDM in solving MI-LP problems can reduce the heavy computation time and memory size to generate the optimal solution.

In future studies, we consider a combination of CDM with heuristic methods, for example, the method proposed in \cite{safarina2017conic}. In particular, a feasible value obtained by the method of \cite{safarina2017conic} would give a good tentative value in the framework of branch-and-bound for solving MI-LP ($P^k$). Another direction is that the decomposition in CDM and the generation of linear cuts can be used not only to solve the OCS problem in tree breeding, but also can be applied to other MI-SOCP problems. We will also consider another problem of OCS that involves not only simple binary constraints but also other types of integer constraints.

\section{Acknowledgement}
Our work was partially supported by funding from
JSPS KAKENHI (Grant-in-Aid for Scientific Research (C), 15K00032).

\bibliographystyle{abbrv}
\bibliography{reference}

\end{document}